\documentclass[11pt]{article}
\usepackage{amsfonts}
\usepackage{amsmath}

\newtheorem{theorem}{Theorem}

\newtheorem{lemma}[theorem]{Lemma}

\newtheorem{remark}[theorem]{Remark}

\newenvironment{proof}[1][Proof]{\noindent\textbf{#1.} }{\ \rule{0.5em}{0.5em}}
\input{tcilatex}
\begin{document}

\title{Trivial unit conjecture and homotopy theory}
\author{Shengkui Ye}
\maketitle

\begin{abstract}
A homotopy theoretic description is given for trivial unit conjecture in the
group ring $\mathbb{Z}G.$
\end{abstract}

\section{Introduction}

Let $G$ be a torsion-free group and $\mathbb{Z}G$ the integral group ring.
The trivial unit conjecture for $G$ says that any invertible element (unit)
of $\mathbb{Z}G$ is of the form $\pm g$ for some $g\in G$ (cf. \cite{Pa},
Chapter 13). For solving such a conjecture, to the author's knowledge,
almost all the approaches used are algebraic (cf. \cite{CP} and references
therein). In this note, we give a homotopy theoretic description of such a
conjecture.

Let $X$ be a CW complex with fundamental group $\pi _{1}(X)=G.$ For any
integer $d\geq 2$ and map $f:S^{d}\rightarrow X\vee S^{d}$, we construct a
CW complex $Y_{f}=(X\vee S^{d})\cup _{f}e^{d+1}$. In this note, the
following homotopy theoretic characterization is obtained:

\begin{theorem}
\label{main}Let $G$ be a torsion-free group. The trivial unit conjecture for 
$G$ is true if and only if for an Eilenberg-Mac Lane space $X=BG,$ the
element $[f]\in \pi _{d}(\widetilde{X\vee S^{d}},S^{d})$ (the relative
homotopy group of the universal covering space) vanishes for some lifting of 
$S^{d}$ whenever the inclusion $i_{f}:X\rightarrow Y_{f}$ is a homotopy
equivalence.
\end{theorem}

All modules considered in this note are left modules. Let $\tilde{Y}_{f}$ be
the universal covering space of $Y_{f}$ and $C_{i}(\tilde{Y}_{f})$ the $i$%
-th term of the cellular chain complex of $\tilde{Y}_{f}.$ By definition, $%
C_{i}(\tilde{Y}_{f})$ is a free $\mathbb{Z}G$-module spanned by the set of
all $i$-cells. For the inclusion $i_{f}:X\rightarrow Y_{f},$ we have a
cellular map $\tilde{\imath}_{f}:\tilde{X}\rightarrow \tilde{Y}_{f}$ which
lifts $i_{f}.$ As the map $i_{f}$ induces the identity homomorphism on
fundamental groups of $X$ and $Y_{f},$ we may assume that $\tilde{X}$ is a
subspace of $\tilde{Y}_{f}.$ The relative chain complex $C_{\ast }(\tilde{Y}%
_{f},\tilde{X})$ of $(\tilde{Y}_{f},\tilde{X})$ is of the following form%
\begin{equation*}
0\rightarrow C_{d+1}(\tilde{Y}_{f},\tilde{X})=\mathbb{Z}G\overset{\partial }{%
\rightarrow }C_{d}(\tilde{Y}_{f},\tilde{X})=\mathbb{Z}G\rightarrow 0.
\end{equation*}%
This is a chain complex whose terms are all vanishing except for the $d$-th
term a free $\mathbb{Z}G$-module spanned by $S^{d}$ and the $(d+1)$-th term
a free $\mathbb{Z}G$-module spanned by $e^{d+1}.$ Let $\gamma _{f}=\partial
(1)\in \mathbb{Z}G,$ the unique element determined by the boundary map $%
\partial .$ We give a homotopy theoretic description of units in $\mathbb{Z}%
G $ as follows.

\begin{lemma}
\label{sec}Let $\gamma _{f}\in \mathbb{Z}G$ be the element defined above$.$
Then $\gamma _{f}$ is an invertible element if and only if the inclusion $%
i_{f}:X\hookrightarrow Y_{f}$ is a homotopy equivalence.
\end{lemma}

\begin{proof}
All the notations used in this proof are the same as defined before. Suppose
that $\gamma _{f}=\partial (1)$ is an invertible element in $\mathbb{Z}G$.
Then $\partial $ is both injective and surjective, which shows the relative
chain complex $C_{\ast }(\tilde{Y}_{f},\tilde{X})$ is acyclic. This implies
that $\tilde{\imath}_{f}$ induces an isomorphism between the homology groups 
$H_{i}(\tilde{X})$ and $H_{i}(\tilde{Y}_{f})$ for each $i\geq 0.$ Since $%
\tilde{X}$ and $\tilde{Y}_{f}$ are both simply connected, $\tilde{\imath}%
_{f}:\tilde{X}\rightarrow \tilde{Y}_{f}$ is a homotopy equivalence. Since $%
i_{f}$ induces the identity homomorphism on fundamental groups, this shows
that $i_{f}:X\rightarrow Y_{f}$ is a homotopy equivalence by the Whitehead
theorem.

Conversely, suppose that $i_{f}:X\rightarrow Y_{f}$ is a homotopy
equivalence. Then $\tilde{\imath}_{f}:\tilde{X}\rightarrow \tilde{Y}_{f}$ is
a homotopy equivalence, which implies that the relative chain complex $%
C_{\ast }(\tilde{Y}_{f},\tilde{X})$ is acyclic. This implies that $\gamma
_{f}=\partial (1)$ has a left inverse. It is a well-known fact that in the
integral group ring of a torsion-free group, one-sided invertible element is
also two-sided invertible (cf. Corollary 1.9 from \cite{Pa}, p.38). This
finishes the proof.
\end{proof}

\bigskip

\begin{proof}[Proof of Theorem \protect\ref{main}]
Let $X=BG,$ the classifying space of $G.$ Suppose that the trivial unit
conjecture for $G$ is true. For an integer $d\geq 2$ and a map $%
f:S^{d}\rightarrow X\vee S^{d},$ suppose that the CW complex $Y_{f}=(X\vee
S^{d})\cup _{f}e^{d+1}$ has its inclusion $i_{f}:X\rightarrow Y_{f}$ a
homotopy equivalence. By Lemma \ref{sec}, the element $\gamma _{f}$ is a
unit. Therefore, $\gamma _{f}=\pm g$ for some element $g\in G$. As the $d$%
-th and $(d+1)$-th terms of the relative chain complex are free $\mathbb{Z}G$%
-modules, we can view them as submodules of $C_{i}(\tilde{Y})$ $(i=d,d+1$
resp.$).$ Since $\tilde{X}$ is a free $G$-CW complex and $S^{d}$ is simply
connected, the universal covering space $\widetilde{X\vee S^{d}}$ could be
taken as the push out the following diagram%
\begin{equation*}
\begin{array}{ccc}
G\times \mathrm{pt} & \rightarrow & \tilde{X} \\ 
\downarrow &  & \downarrow \\ 
G\times S^{d} & \rightarrow & \tilde{X}\vee _{G}(G\times S^{d}).%
\end{array}%
\end{equation*}%
Since $X=BG$ is aspherical, $\tilde{X}$ is contractible. This implies that
there is a homotopy equivalence $\tilde{X}\vee _{G}(G\times S^{d})\simeq
\vee _{G}S^{d},$ where $\vee _{G}S^{d}$ is the wedge of copies of $S^{d}$
indexed by $G.$ For any element $h\in G,$ let $p_{h}:\widetilde{X\vee S^{d}}$
$\rightarrow S^{d}$ be the projection onto the $h$-component of $\vee
_{G}S^{d}.$ Consider a lifting $\tilde{f}$ of $f$ to the universal covering
space as shown in the following diagram%
\begin{equation*}
\begin{array}{ccc}
&  & \widetilde{X\vee S^{d}} \\ 
& \tilde{f}\nearrow & \downarrow \\ 
S^{d} & \overset{f}{\rightarrow } & X\vee S^{d}.%
\end{array}%
\end{equation*}%
This $\tilde{f}$ actually determines the $(d+1)$-\textrm{th} boundary map in
the chain complex of $\tilde{Y}_{f}.$ By the definition of the boundary map $%
\partial $, the degree of the composition%
\begin{equation*}
S^{d}\overset{\tilde{f}}{\rightarrow }\widetilde{X\vee S^{d}}\overset{p_{h}}{%
\rightarrow }S^{d}
\end{equation*}%
is zero when $h\neq g$ or $\pm 1$ when $h=g.$ Therefore, $\tilde{f}$ is
homotopic to some map $\tilde{g}$ whose image occupies only the $g$%
-component $S^{d}.$ This shows that $[f]:=[\tilde{f}]\in \pi _{d}(\widetilde{%
X\vee S^{d}},S^{d})$ is vanishing, where $S^{d}$ is viewed as the $g$%
-component $S^{d}$.

Conversely, suppose that $\gamma $ is a nontrivial invertible element in $%
\mathbb{Z}G.$ We will construct some map $f_{\gamma }:S^{d}\rightarrow X\vee
S^{d}$ such that the inclusion $i_{f_{\gamma }}:X\rightarrow Y_{f}$ is a
homotopy equivalence but $[f_{\gamma }]\in \pi _{d}(\widetilde{X\vee S^{d}}%
,S^{d})$ is not vanishing for any lifting of $S^{d}$. Assume that $\gamma
=\sum a_{g}g$ for $g\in G$ and $a_{g}\in \mathbb{Z}$. As in the first part
of this proof, the universal covering space $\widetilde{X\vee S^{d}}=\tilde{X%
}\vee _{G}(G\times S^{d})$ could be a free $G$-CW complex. Let $p_{h}:\tilde{%
X}\vee _{G}(G\times S^{d})\simeq \vee _{G}S^{d}\rightarrow S^{d}$ be the
projection onto the $h$-component. Define $\tilde{f}_{\gamma
}:S^{d}\rightarrow \tilde{X}\vee _{G}(G\times S^{d})\simeq \vee _{G}S^{d}$
as a cellular map such that the degree of the composition%
\begin{equation*}
S^{d}\overset{\tilde{f}_{\gamma }}{\rightarrow }\widetilde{X\vee S^{d}}%
\overset{p_{h}}{\rightarrow }S^{d}
\end{equation*}%
is $a_{h}$ for each $h\in G.$ Denote by 
\begin{equation*}
\phi _{\gamma }:G\times S^{d}\rightarrow \tilde{X}\vee _{G}(G\times S^{d})
\end{equation*}%
the unique $G$-equivariant map determined by $\tilde{f}_{\gamma }.$ Note
that $\phi _{\gamma }$ is a $G$-equivariant between two free $G$-CW
complexes. Passing to the quotient space, we get a map $f_{\gamma
}:S^{d}\rightarrow \tilde{X}\vee _{G}(G\times S^{d})/G=X\vee S^{d}$ such
that the following diagram is commutative%
\begin{equation*}
\begin{array}{ccc}
&  & \widetilde{X\vee S^{d}} \\ 
& \tilde{f}_{\gamma }\nearrow  & \downarrow  \\ 
S^{d} & \overset{f_{\gamma }}{\rightarrow } & X\vee S^{d}.%
\end{array}%
\end{equation*}%
Construct a free $G$-CW complex $\widetilde{Y}_{\gamma }=\widetilde{X\vee
S^{d}}\cup _{\phi _{\gamma }}(G\times e^{d+1})$ as the push out of the
following diagram%
\begin{equation*}
\begin{array}{ccc}
G\times S^{d} & \overset{\phi _{\gamma }}{\rightarrow } & \tilde{X}\vee
_{G}(G\times S^{d}) \\ 
\downarrow  &  & \downarrow  \\ 
G\times e^{d+1} & \rightarrow  & \widetilde{Y_{f}}.%
\end{array}%
\end{equation*}%
This $G$-CW complex $\widetilde{Y_{\gamma }}$ is actually the universal
cover of $Y_{\gamma }:=\tilde{Y}_{\gamma }/G$ (for more details on the
construction, see the proof of Lemma 2.2 in \cite{Lu1} or p.371 in \cite{Lu2}%
). According to Lemma \ref{sec}, the inclusion $i_{f}:X\rightarrow Y_{\gamma
}$ is a homotopy equivalence, since $\gamma $ is a unit. Let $%
i_{g}:S^{d}\hookrightarrow \widetilde{X\vee S^{d}}=\tilde{X}\vee
_{G}(G\times S^{d})$ be the inclusion of $S^{d}$ into the $g$-component. As $%
\gamma $ is nontrivial, the map $\tilde{f}_{\gamma }$ is not homotopic to
any map $S^{d}\rightarrow S^{d}\overset{i_{g}}{\hookrightarrow }\widetilde{%
X\vee S^{d}}=\tilde{X}\vee _{G}(G\times S^{d})$ for any $g\in G$ by
considering the degree of $p_{h}\tilde{f}_{\gamma }$ for each $h\in G.$ This
shows that $[f_{\gamma }]:=[\tilde{f}_{\gamma }]\in \pi _{d}(\widetilde{%
X\vee S^{d}},S^{d})$ is not vanishing for any lifting of $S^{d}$.
\end{proof}

\begin{remark}
For zero divisor conjecture in $\mathbb{Z}G,$ some necessary conditions of
homotopy descriptions are given in \cite{Iv} and \cite{Le}.
\end{remark}

\noindent \textbf{Acknowledgement}

This note was finished when the author was a Phd student in National
University of Singapore (NUS). He is grateful to his thesis advisor
Professor A.J. Berrick for many discussions and constant encouragement.

\bigskip

Mathematics and Physics Centre, Xi'an Jiaotong-Liverpool University, 111 Ren
Ai Road, Suzhou, Jiangsu 215123, China.

E-mail: Shengkui.Ye@xjtlu.edu.cn

Mathematical Institute, University of Oxford, 24-29 St Giles', Oxford, OX1
3LB, U.K.

E-mail: Shengkui.Ye@maths.ox.ac.uk


\bigskip

\begin{thebibliography}{Lu1}
\bibitem[Cp]{CP} D.A. Craven and P. Pappas, On the Unit Conjecture for
Supersoluble Group Rings, I, arXiv:1010.1144.

\bibitem[Iv]{Iv} S. V. Ivanov, An asphericity conjecture and Kaplansky
problem on zero divisors, J. Algebra 216 (1999), no. 1, 13--19.

\bibitem[Le]{Le} I.J. Leary, Asphericity and zero divisors in group
algebras, Journal of Algebra 227(2000), 362-364.

\bibitem[Lu1]{Lu1} W. L\"{u}ck, $L^{2}$ invariants of regular coverings of
compact manifolds and CW-complexes, Handbook of geometric topology",
editors: Davermann, R.J. and Sher, R.B., Elsevier, 2002.

\bibitem[Lu2]{Lu2} W. L\"{u}ck, $L^{2}$\textit{-Invariants: theory and
applications to geometry and }$K$\textit{-Theory}, Ergebnisse der Mathematik
und ihrer Grenzgebiete 44, Springer, 2002.

\bibitem[Pa]{Pa} D. S. Passman, The algebraic structure of group rings. John
Wiley \& Sons, New York, London, Sydney, Toronto, 1977.
\end{thebibliography}
\end{document}